\documentclass{amsart}
\usepackage{amsmath,mathtools}
\usepackage{amsrefs}
\usepackage{microtype}
\usepackage{lmodern}
\usepackage{hyperref}
\usepackage{tikz-cd}
\usepackage{graphicx}
\usepackage{tikz-network}
\usetikzlibrary{arrows}
\usepackage{enumerate}
 
\usetikzlibrary{decorations.markings}
\usepackage{caption}
\theoremstyle{plain}

\newtheorem{lemma}[subsection]{Lemma}
\newtheorem{corollary}[subsection]{Corollary}

\theoremstyle{definition}
\newtheorem{definition}[subsection]{Definition}

\theoremstyle{remark}
\newtheorem{remark}[subsection]{Remark}

\numberwithin{equation}{section}

\newcommand{\cal}{\mathcal}
\newcommand{\cla}{{\cal A}}
\newcommand{\clb}{{\cal B}}

\newcommand{\cld}{{\cal D}}
\newcommand{\cle}{\cal{E}}

\newcommand{\clg}{{\cal G}}
\newcommand{\clh}{{\cal H}}

\newcommand{\cll}{{\cal L}}

\newcommand{\der}{\textup{d}}
\newcommand{\nquad}{\!\!\!}

\def\a*{{\cal A}_{h,*}}
\def\B{{\cal B}(h)}
\def\B1{{\cal B}_1(h)}
\def\b{{\cal B}^{\rm s.a.}(h)}
\def\b1{{\cal B}^{\rm s.a.}_1(h)}

\newcommand{\ot}{ \ \otimes \ }

\newcommand{\raro}{\rightarrow}

\def\a*{{\cal A}_{h,*}}
\def\B1{{\cal B}_1(h)}
\def\b{{\cal B}^{\rm s.a.}(h)}
\def\b1{{\cal B}^{\rm s.a.}_1(h)}

\usepackage{graphicx} 
\begin{document}

\title[]{Twisted Edge Laplacians on finite graphs from a K\"{a}hler structure}
\author[Joardar]{Soumalya Joardar}
\address{Department of Mathematics and Statistics, Indian Institute of Science Education and Research Kolkata, Mohanpur - 741246, West
	Bengal, India}
\email{soumalya@iiserkol.ac.in}
\author[Rahaman]{Atibur Rahaman}
\address{Department of Mathematics and Statistics, Indian Institute of Science Education and Research Kolkata, Mohanpur - 741246, West
	Bengal, India}
\email{atibur.pdf@iiserkol.ac.in}
\begin{abstract}
  This is a continuation of the work done by the authors in \cite{Sou_Ati1}. In \cite{Sou_Ati1} an almost complex structure on finitely many points from bidirected polygon was introduced. In this paper we study a K\"{a}hler structure on finite points. In particular, we study the edge Laplacian of a graph twisted by the K\"{a}hler structure introduced in this paper. We also discuss a metric aspect from a twisted holomorphic Dolbeault-Dirac spectral triple and show that the points have a finite diameter with respect to Connes' distance.   
\end{abstract}
\maketitle
\section{Introduction}
One of the exciting problems in recent times in the context of noncommutative geometry has been to extend the classical complex geometric ideas into the realm of noncommutative geometry. Although the interplay between the algebraic and analytic aspects of noncommutative complex geometry is far from well understood, the study of noncommutative complex geometry has gained some momentum in recent years. The axioms have been mainly made keeping the quantum flag manifolds in sight. The reader can see \cites{Kahler_Buachalla,Somberg, Beggs} for familiarizing himself/herself with such axioms and set up. Although it is worth mentioning that there are other examples and approaches to noncommuatative complex geometry (see \cites{Khalkhali-NCGQspace, Schwartz, Staford, NCTori} for example). After setting up basic complex geometric notions in the realm of the noncommutative geometry, keeping the success of the K\"{a}hler structure in classical complex geometry in mind, the noncommutative K\"{a}hler structure has been explored mainly in the context of quantum flag manifolds by various authors (see \cites{Kahler_Buachalla, Das_Kahler, Kodaira} among many others). It is worth mentioning that whether metrics have a key role to play in noncommutative complex geometry is still not very clear. The study of noncommutative complex geometry being in its relatively early days, it has become imperative to construct new examples to test any hypothesis. The authors of this paper have recently explored an almost complex structure on finitely many points in \cite{Sou_Ati1}. There it has been established that given a bidirected graph over finitely many points, there is a natural almost complex structure associated to it. Furthermore such almost complex structure is naturally integrable. Then a concrete example of such an almost complex structure was studied where the bomodule of one-forms came from a bidirected polygon. In this paper we continue exploring the almost complex structure and we introduce a K\"{a}hler structure. But we don't discuss the Hodge theory here. Rather we study a twisted holomorphic Dolbeault-Dirac operator on a holomorphic hermitian module and an associated twisted edge Laplacian in the spirit of \cite{Das_Kahler}. At this point it is worth mentioning that an almost complex structure has also been explored on a three point space using the spectral triple formalism in \cite{debashish}. But unfortunately there is no compatible K\"{a}hler structure in their set up. Coming back to the twisted Laplacian, the twist arises from a left module on finite points which comes from a directed graph on the points. It is shown that the left module coming from a directed graph is a finitely generated projective module admitting a canonical hermitian structure (see Lemma \ref{Hermitian}). A general holomorphic structure on such a module is indexed by a class of left module maps (see the discussion following Corollary \ref{hermitianpositive}). one can restrict the Laplacian on the corresponding left module and the left module being the function space over the edges, can rightfully termed as a twisted edge Laplacian on the corresponding graph. We explore the spectrum of such a twisted edge Laplacian for regular graphs, in particular regular $k$-gons. Then we go on to consider a twisted holomorphic Dolbeaut-Dirac operator coming from a holomorphic hermitian module. Due to the finite dimensionality of the Hilbert spaces involved the adjointabilty of such a Dirac operator (any other such analytical issue) does not pose any problem and we are content with considering Connes' distance between points coming from the twisted Dolbeault Dirac spectral triple. In particular we show that $n$-many points have finite diameter$\leq[\frac{n}{2}]$ (Corollary \ref{diameter}).   
\section{K\"{a}hler structure on finite points from polygon graph}
We start off by briskly recalling the rudiments of noncommutative complex geometry. We mostly follow the definitions and conventions of \cites{Somberg, Sou_Ati1, Kahler_Buachalla, Kodaira}.
\begin{definition}(\cite{Khalkhali})
Let $\cla$ be a $\ast$-algebra over the field of complex numbers $\mathbb{C}$.
Then a $\ast$-differential calculus over \(\cla\) is a datum $(\Omega^{\bullet}(\cla),\der,\ast)$  where
\begin{enumerate}[(i)]
    \item $\Omega^{\bullet}(\cla)=\oplus_{n\geq 0}\Omega^{n}(\cla)$ is a graded $\ast$-algebra such that $\Omega^{\bullet}(\cla)$ is an $\cla$\ndash$\cla$-bimodule with $\Omega^{0}(\cla)=\cla$;
    \item $\der:\Omega^{\bullet}(\cla)\raro\Omega^{\bullet+1}(\cla)$ is a graded derivation i.e. $\der$ is a $\mathbb{C}$-linear map satisfying the graded Leibniz rule
    \begin{displaymath}
    \der(\omega\wedge\eta)=\der\omega\wedge\eta+(-1)^{{\rm deg}(\omega)}\omega\wedge \der\eta,
    \end{displaymath}
    for $\omega\in \Omega^{n}(\cla),\eta\in \Omega^{\bullet}(\cla)$ and $\der^2=0$;
    \item for all $\omega\in\Omega^{\bullet}(\cla)$, $\der(\omega^{\ast})=(\der\omega)^{\ast}$.
\end{enumerate}
\end{definition} 
\begin{remark}
    As usual practice, we have denoted the product of the graded product by $\wedge$ and with respect to the product $\wedge$, $\ast$ is antilinear graded involution meaning that
    \begin{displaymath}
        (\omega\wedge\eta)^{\ast}=(-1)^{{\rm deg}(\omega){\rm deg}(\eta)}\eta^{\ast}\wedge\omega^{\ast}.
    \end{displaymath}
\end{remark}
\begin{definition}
Let $(\Omega^{\bullet}(\cla),\der,\ast)$ be a $\ast$-differential calculus over a $\ast$-algebra $\cla$. An almost complex structure on $(\Omega^{\bullet}(\cla),\der,\ast)$ is a degree zero derivation $J:\Omega^{\bullet}(\cla)\rightarrow\Omega^{\bullet}(\cla)$ such that
\begin{enumerate}[1.]
    \item $J$ is identically zero on $\cla$ and hence an $\cla$\ndash$\cla$-bimodule endomorphism of $\Omega^{\bullet}(\cla)$;
    \item $J^{2}=-{\rm Id}$ on $\Omega^{1}(\cla)$; and
    \item $J(\omega^{\ast})=(J\omega)^{\ast}$ for all $\omega\in\Omega^{1}(\cla)$.
\end{enumerate}
\end{definition}
Given an almost complex structure on a $\ast$-differential calculus, the bimodule of one-forms has the following decomposition into eigenspaces of $J$:\begin{displaymath}
\Omega^{1}(\cla)=\Omega^{1,0}(\cla)\oplus\Omega^{0,1}(\cla),
\end{displaymath}
where $\Omega^{1,0}(\cla):=\{\omega\in\Omega^{1}(\cla): J(\omega)=i\omega\}$, $\Omega^{0,1}:=\{\omega\in\Omega^{1}(\cla): J(\omega)=-i\omega\}$. $\Omega^{1,0}(\cla)$ and $\Omega^{0,1}(\cla)$ are again $\cla$\ndash$\cla$-bimodules and called the bimodules of $(1,0)$ and $(0,1)$-forms respectively. More generally, one has the higher dimensional analogues of the $(1,0)$ and $(0,1)$-forms. For a fixed $n$, the bimodule $\Omega^{n}(\cla)$ has the following decomposition into bimodules of $(p,q)$-forms for $p+q=n$, where $p,q$ are non negative integers(\cite{Beggs}*{Sec. 2.5}): 
\begin{displaymath}
\Omega^{n}(\cla)=\bigoplus_{p+q=n}\Omega^{p,q}(\cla),
\end{displaymath}
where $\Omega^{p,q}(\cla):=\{\omega\in\Omega^{n}(\cla):J(\omega)=i(p-q)\omega\}$. In the following definition we are going to use the Lefschetz operators. For the definition of the Lefschetz operators the reader is referred to \cite{Kahler_Buachalla} (see Definition 4.1).
\begin{definition}
    Given an almost complex structure $(\Omega^{\bullet}(\cla), \der,\ast, J)$, a K\"{a}hler form $\kappa$ is a central closed $(1,1)$ form such that with respect to Lefschetz map $L:\Omega^{\bullet}(\cla)\raro\Omega^{\bullet}(\cla)$ given by $\omega\mapsto \kappa\wedge\omega$, isomorphisms are given by
    \begin{displaymath}
        L^{n-k}: \Omega^{k}\raro\Omega^{2n-k}, \ for \ all \ 0\leq k< n.
    \end{displaymath}
\end{definition}
Given a K\"{a}hler-structure with the associated K\"{a}hler-form $\kappa$, one can define the Hodge $\ast$-map $\ast_{\kappa}$ which we do not recall here. The reader is referred to Section 2.3 of \cite{Das_Kahler} for details of the Hodge $\ast$-map on a K\"{a}hler structure. Now a metric can be defined using the Hodge star map.
\begin{definition}
    \label{metric}
    The canonical Hermitian metric associated to the K\"{a}hler structure $(\Omega^{\bullet}(\cla),\der,J,\kappa)$ is the unique map $g_{\kappa}:\Omega^{\bullet}(\cla)\times\Omega^{\bullet}(\cla)\raro\cla$ such that $g_{\kappa}(\Omega^{k},\Omega^{l})=0$ for $k\neq l$ and 
    \begin{displaymath}
        g_{\kappa}(\omega,\eta)=\ast_{\kappa}(\omega\wedge \ast_{\kappa}(\eta^{\ast})).
    \end{displaymath}
    The Hermitian metric $g_{\kappa}$ is said to be {\it positive definite} if $g_{\kappa}(\omega,\omega)$ belongs to the positive cone $\cla^{+}:={\rm Sp}\{a^{\ast}a:a\in\cla\}$ for all $\omega\in\Omega^{\bullet}(\cla)$.
\end{definition}
We are interested in a K\"{a}hler structure on finitely many points. For that purpose we recall the almost complex structure coming from the bidirected polygon on $n$-many vertices from \cite{Sou_Ati1}. We follow the notations and conventions as in \cite{Sou_Ati1}. So we label the vertices by natural numbers modulo $n$. We denote the function on $E$ which takes the value $1$ on an edge $\mu\raro\nu$ and zero on the rest of the vertices by $\xi_{\nu\raro\nu}$. Then recall that with respect to a canonical prolongation procedure using a bimodule map $\sigma$, the differential calculus coming from the polygon is a $2$-dimensional orientable calculus. It has an almost complex structure which determines the following spaces of forms:
\begin{eqnarray*}
    &&\Omega^{1,0}(\cla)={\rm Sp}_{\mathbb{C}}\{\xi_{\mu\raro\mu+1}\}_{\mu=1,2,\ldots,n}\\
    &&\Omega^{0,1}(\cla)={\rm Sp}_{\mathbb{C}}\{\xi_{\mu+1\raro\mu}\}_{\mu=1,2,\ldots,n}\\
    &&\Omega^{2,0}(\cla)=\Omega^{0,2}(\cla)=0\\
    &&\Omega^{1,1}(\cla)={\rm Sp}_{\mathbb{C}}\{\xi_{\mu\raro\mu-1}\wedge\xi_{\mu-1\raro\mu}\}_{\mu=1,2,\ldots,n}.
\end{eqnarray*}
\begin{lemma}
    \label{Kahler form}
    On the almost complex structure on finitely many points coming from the bidirected polygon graph the form $\kappa=i\sum_{\mu\in V}\xi_{\mu\raro\mu-1}\wedge\xi_{\mu-1\raro\mu} $ is a K\"{a}hler form. 
\end{lemma}
\begin{proof}
     By construction $\kappa$ is an $(1,1)$-form and central. As $\Omega^{3}(\cla)=0$, $\kappa$ is closed i.e. $d\kappa=0$. Recalling that for any edge $\mu\raro\nu$, $\xi_{\mu\raro\nu}^{\ast}=-\xi_{\nu\raro\mu}$ and $\xi_{\mu\raro\mu-1},\xi_{\mu-1\raro\mu}\in\Omega^{1}(\cla)$, a straightforward checking establishes that $\kappa$ is real i.e. $\kappa^{\ast}=\kappa$. As for the isomorphisms $L_{\kappa}^{n-k}:\Omega^{k}\raro\Omega^{2n-k}$, as $n=1$, there is only one isomorphism to settle viz. $L_{\kappa}:\cla\raro\Omega^{2}(\cla)$ given by $f\mapsto f\kappa$. But this is contained in the proof of Theorem 3.2 of \cite{Sou_Ati1}. This establishes that $\kappa$ is a K\"{a}hler form. 
\end{proof}
Now working through various definitions from \cite{Kahler_Buachalla} (in particular Definition 4.11), one can write down the formulae for the Hodge $\ast$-map on forms. This is relatively easy and we collect them in the following lemma leaving the straightforward proof to the reader. The reader is referred to \cite{Kahler_Buachalla} (see Definition 4.2) for the definition of the spaces of primitive forms $P^{(a,b)}$ which are defined using the Lefschetz map $L$. 
\begin{lemma}
    \label{Hodgestarformula}
    With the notations of \cite{Kahler_Buachalla}, the spaces $P^{(a,b)}$ are given by $P^0=\cla$, $P^{1,0}=\Omega^{1,0}(\cla)$, $P^{0,1}=\Omega^{0,1}(\cla)$, $P^{2}=0$.
    The Hodge $\ast$-map $\ast_{\kappa}$ is given by the following formulae:\begin{eqnarray}
        && \ast_{\kappa}(f)=L_{\kappa}(f)=f\kappa \ for \ f\in\cla\\
        && \ast_{\kappa}(L_{\kappa}^{0}(\omega))=\ast_{\kappa}(\omega)=-i\omega \ for \ \omega\in \Omega^{1,0}(\cla)\\
        && \ast_{\kappa}(L_{\kappa}^{0}(\omega))=\ast_{\kappa}(\omega)=i\omega \ for \ \omega\in\Omega^{0,1}(\cla)\\
        && \ast_{\kappa}(\omega)=f \ for \ \omega\in\Omega^{2}(\cla) \ with \ L_{\kappa}^{-1}(\omega)=f.
    \end{eqnarray}
\end{lemma} 
\begin{lemma}
    \label{positivedefinite}
    The canonical Hermitian metric $g_{\kappa}$ corresponding to the K\"{a}hler form $\kappa$ is positive definite.
\end{lemma}
\begin{proof}
    Clearly it suffices to show that for any $\omega\in\Omega^{k}(\cla)$, $g_{\kappa}(\omega,\omega)\in\cla_{+}$ for $k=0,1,2$. For $k=0$ and $f\in\cla$, $g_{\kappa}(f,f)=\ast_{\kappa}(f\wedge\ast_{\kappa}(f^{\ast}))=|f|^2\in\cla^{+}$.  For an $\omega=\sum c_{\mu}\xi_{\mu\rightarrow\mu+1}\wedge\xi_{\mu+1\rightarrow\mu}\in\Omega^{1,1}(\cla)$, 
        \begin{eqnarray*}
            g_{\kappa}(\omega,\omega)&=& \ast_{\kappa}(\omega\wedge\ast_{\kappa}(\omega^{\ast}))\\
            &=& f,
        \end{eqnarray*}
        where $f(\mu)=|c_{\mu}|^2$ so that $f\in\cla^{+}$. As $\Omega^{2,0}(\cla)=\Omega^{0,2}(\cla)=0$, we have $g_{\kappa}(\omega,\omega)\in\cla^{+}$ for all $\omega\in\Omega^{2}(\cla)$.\\
        \indent Now note that for any $\omega\in\Omega^{1,0}(\cla)$ and $\eta\in \Omega^{0,1}(\cla)$, $g_{\kappa}(\omega,\eta)=0=g_{\kappa}(\eta,\omega)$. With this observation for one forms it is enough to show that $g_{\kappa}(\omega,\omega)\in\cla^{+}$ for $\omega\in\Omega^{1,0}(\cla)$ or $\omega\in\Omega^{0,1}(\cla)$. So let $\omega=\sum{c_{\mu}}\xi_{\mu\rightarrow\mu+1}\in\Omega^{1,0}(\cla)$ where $c_{\mu}\in\mathbb{C}$, 
    \begin{eqnarray*}
        g_{\kappa}(\omega,\omega)&=&\ast_{\kappa}(\omega\wedge\ast_{\sigma}(\omega^{\ast}))\\
        &=& \ast_{\kappa}(-i\sum |c_{\mu}|^{2}\xi_{\mu\rightarrow\mu+1}\wedge\xi_{\mu+1\rightarrow\mu})\\
        &=& f,
    \end{eqnarray*}
    where $f(\mu)=|c_{\mu}|^2$ and therefore $g_{\kappa}(\omega,\omega)\in\cla^{+}$. For $\omega=\sum c_{\mu}\xi_{\mu\rightarrow\mu-1}\in\Omega^{0,1}(\cla)$,
    \begin{eqnarray*}
        g_{\kappa}(\omega,\omega)&=&\ast_{\kappa}(\omega\wedge\ast_{\kappa}(\omega^{\ast}))\\
        &=& \ast_{\kappa}(i\sum |c_{\mu}|^{2}\xi_{\mu\rightarrow\mu-1}\wedge\xi_{\mu-1\rightarrow\mu})\\
        &=& f,
        \end{eqnarray*}
        where $f(\mu)=|c_{\mu}|^2$ and therefore $g_{\kappa}(\omega,\omega)\in\cla^{+}$ for all $\omega\in\Omega^{1}(\cla)$.
\end{proof}
\section{The twisted edge Laplacian}
\label{edgelaplacian}
In this section we shall consider Hermitian holomorphic modules over finitely many points. In particular we shall focus on such modules arising from finite directed graphs on finitely many points. To that end we start by recalling the definitions of holomorphic and hermitian modules over a $\ast$-algebra.
\begin{definition}
    \label{holomorphic}
   Let $\cla$ be a $\ast$-algebra admitting an almost complex structure. A holomorphic module over $\cla$ is a pair $(\cle,\bar{\partial}_{\cle})$ where $\cle$ is a finitely generated projective left module over $\cla$ and $\bar{\partial}_{\cle}:\cle\raro \Omega^{0,1}(\cla)\bar{\ot}\cle$ is a flat $(0,1)$-connection.
\end{definition}
In the following definition and throughout the rest of the paper, for a left module $\cle$ over an algebra $\cla$, we denote the set of all left module maps from $\cle\raro\cla$ by $^\vee\nquad\mathcal{E}$ which has a natural right module structure (see \cite{Das_Kahler}). For the next definition the reader is urged to consult section 2.6 of \cite{Das_Kahler} for the definition (as well as the right module structure) of the conjugate module $\bar{\cle}$.
\begin{definition}
    \label{hermitian}
    A hermitian module over an algebra $\cla$ is a pair $(\cle, h_{\cle})$ where $\cle$ is a finitely generated projective left module and a right $\cla$-isomorphism $h_{\cle}:\bar{\cle}\raro ^\vee\nquad\mathcal{E}$ where $\bar{\cle}$ is the conjugate right $\cla$-module such that for the associated sesquilinear pairing $h_{\cle}(-,-):\cle\times\cle\raro\cla$ given by $(e,k)\mapsto h_{\cle}(\bar{k})(e)$ satisfies 
    \begin{displaymath}
        h_{\cle}(e,k)=h_{\cle}(k,e)^{\ast} \quad \text{ and }\quad h_{\cle}(e,e)\in\cla^{+}.
    \end{displaymath}
\end{definition}
Given a holomorphic, hermitian module over a $\ast$-algebra $\cla$ admitting a K\"{a}hler structure $(\Omega^{\bullet}(\cla),\der,J,\kappa)$, recall the map $g_{\cle}:\Omega^{\bullet}(\cla)\bar{\ot}\cle\times\Omega^{\bullet}(\cla)\bar{\ot}\cle\raro\cla$ from \cite{Das_Kahler} (see discussion after Definition 2.3). If the $\ast$-algebra $\cla$ admits a faithful state $\tau:\cla\raro\mathbb{C}$, then there is a $\mathbb{C}$-valued inner product $\langle-,-\rangle_{\cle}$ on $\Omega^{\bullet}(\cla)\bar{\ot}\cle$ given by $\tau\circ g_{\cle}$ provided $g_{\cle}$ is positive definite.
\begin{lemma}
\label{Hermitian}
Every left module coming from a finite directed graph on finite points is a finitely generated, projective {\it Hermitian} module. Moreover, the Hermitian structure is positive definite.
\end{lemma}
\begin{proof}
Let us denote the left $C(V)$-module $C(E)$ by $\cle$. For any finite graph $(V,E)$, we denote the function on $E$ which takes the value $1$ on a particular edge $e\in E$ and $0$ on other edges by $\xi_{e}$. Then $\{\xi_{e}\}_{e\in E}$ forms a $\mathbb{C}$-linear basis of $C(E)$. As $\cle$ is a finite dimensional complex vector space and therefore finitely generated as a left module. To prove that $\cle$ is projective, note that the by Prposition 3.2 of \cite{Majid} complete graph on the vertex set $V$ is a free left module over $C(V)$. If the number of vertices is $n$, let us denote the complete graph on $V$ by $K_{n}$ and the corresponding edge set by $E_{n}$. We define a surjective idempotent $I: C(E_{n})\raro\cle$ by the following formula:
\begin{displaymath}
    I(\sum_{e\in E_{n}}\lambda_{e}\xi_{e})=\sum_{e^{\prime}\in E}\lambda_{e^{\prime}}\xi_{e^{\prime}}.
\end{displaymath} Clearly $I$ is a left module morphism and $I(C(E_{n}))\cong\cle$ as left modules and $C(E_{n})\cong \cle\oplus (1-I)(C(E_{n}))$, proving that $\cle$ is a projective module. Now we shall define a canonical Hermitian structure on $\cle$. To that end let us define a right $C(V)$-module isomorphism $h:\bar{\cle}\rightarrow ^\vee\nquad\mathcal{E}$ satisfying the conditions Definition \ref{hermitian}. We claim that $^\vee\nquad\mathcal{E}$ is a finite dimensional vector space with dimension equal to the number of edges in the graph $(V,E)$. For $e\in E$ and $\Phi\in ^\vee\nquad\mathcal{E}$, we have 
\begin{displaymath}
    \Phi(\xi_{e})=\sum_{v\in V}\lambda_{v}\delta_{v}.
\end{displaymath}
For any $w\neq s(e)$, $\delta_{w}\xi_{e}=0$ and consequently $\Phi(\delta_{w}\xi_{e})=0$. On the other hand, as $\Phi$ is a left $C(V)$-module map, 
\begin{eqnarray*}
    \Phi(\delta_{w}\xi_{e})&=& \delta_{w}\Phi(\xi_{e})\\
    &=& \delta_{w}\sum_{v\in V}\lambda_{v}\delta_{v}\\
    &=& \lambda_{w}\delta_{w}.
\end{eqnarray*}
Therefore for any vertex $w\neq s(e)$, $\lambda_{w}=0$. Hence $\Phi(\xi_{e})=\lambda_{s(e)}\delta_{s(e)}$ for some $\lambda_{s(e)}\in\mathbb{C}$. For each $e\in E$, we define $\Phi^{e}\in^\vee\nquad\mathcal{E}$ by the following:
\begin{eqnarray*}
    \Phi^{e}(\xi_{f}):=\delta_{e,f}\delta_{s(e)}.
\end{eqnarray*}
Then it can be shown that $\{\Phi^{e}:e\in E\}$ is a $\mathbb{C}$-linear basis for $^\vee\nquad\mathcal{E}$. Now we define $h:\bar{\cle}\rightarrow ^\vee\nquad\mathcal{E}$ by
$h(\bar{\xi_{e}})=\Phi^{e}$. Clearly this is a vector space isomorphism. To show that this is a right $C(V)$-isomorphism, it suffices to show that $h$ is a right $C(V)$-linear map. To that end let $\phi\in C(V)$. Then $h(\bar{\xi_{e}}.\phi)=h(\bar{\phi^{\ast}\xi_{e}})=\phi(s(e))\Phi^{e}$.
But for any $f\in E$, $(\Phi^{e}.\phi)(\xi_{f})=\delta_{e,f}\phi\delta_{s(e)}$. On the other hand,
\begin{displaymath}
    \phi(s(e))\Phi^{e}(\xi_{f})=\delta_{e,f}\phi(s(e))\delta_{s(e)},
\end{displaymath}
and therefore \begin{displaymath}
    h(\bar{\xi_{e}}.\phi)=\Phi^{e}.\phi=h(\bar{\xi_{e}}).\phi.
\end{displaymath}
 Let $\xi=\sum_{e\in E}\lambda_{e}\xi_{e}, \ \eta=\sum_{e\in E}\mu_{e}\xi_{e}\in C(E)$. Then by definition of the seequiliniear pairing $h_{\cle}(-,-)$,
\begin{eqnarray*}
    h_{\cle}(\xi,\eta)&:=& h(\bar{\eta})(\xi)\\
    &\ =& \sum_{e}\bar{\mu_{e}}\lambda_{e}\delta_{s(e)}.
\end{eqnarray*}
Now it is clear that $h_{\cle}(\xi,\eta)=h_{\cle}(\eta,\xi)^{\ast}$ and $h_{\cle}(\xi,\xi)=\sum_{e\in E}|\lambda_{e}|^2 \delta_{s(e)}\in C(V)^{+}$.
\end{proof}
 Combining Lemma~\ref{positivedefinite} and Lemma~\ref{Hermitian} with Proposition 5.11 of \cite{Kodaira}, we get the following 
\begin{corollary}
\label{hermitianpositive}
 The Hermitian metric $g_{\cle}$ is positive definite.  
\end{corollary}
Recall the canonical faithful state $\tau$ on $C(V)$ given by $\tau(f):=\frac{1}{n}\sum_{\mu\in V}f(\mu)$. Then we have the $\mathbb{C}$-valued inner product (see \cite{Kodaira}) $\langle,\rangle_{\cle}$ on $(\Omega^{\bullet}(\cla)\bar{\ot}\cle)$ given by $\tau\circ g_{\cle}$.\\
\indent Now we shall consider holomorphic modules over $C(V)$ coming from finite directed graph $\clg=(V,E)$. We denote the left module coming from any finite directed graph over the vertex set $V$ by $\cle$. We consider the holomorphicity with respect to the almost complex structure on $V$ coming from the bidirected polygon over $V$. Note that since $\cle$ is finitely generated projective and $\Omega^{0,2}(\cla)=0$, any left connection $\bar{\partial}_{\cle}:\cle\raro\Omega^{0,1}(\cla)\bar{\ot}\cle$ gives a holomorphic module $(\cle,\bar{\partial}_{\cle})$ canonically. Recall that a left connection $\bar{\partial}_{\cle}$ is a $\mathbb{C}$-linear map from $\cle$ to $\Omega^{0,1}(\cla)\bar{\ot}\cle$ such that
\begin{displaymath}
    \bar{\partial}_{\cle}(a.\omega)=\bar{\partial}_{\cle}(a)\ot\omega+a\bar{\partial}_{\cle}(\omega)
\end{displaymath} for $a\in\cla$ and $\omega\in\cle$. Also
recall the element $\theta=\sum_{\mu\in V}\xi_{\mu+1\raro\mu}\in\Omega^{0,1}(\cla)$ from \cite{Sou_Ati1}. Then we have a base $(0,1)$-connection $\nabla_{0}$ on $\cle$ given by the following:
\begin{displaymath}
    \nabla_{0}(\omega)=\theta\ot\omega, \ \omega\in\cle.
\end{displaymath}
 Any $(0,1)$-connection $\bar{\partial}_{\cle}$ is determined by a left module map $\zeta:\cle\raro\Omega^{0,1}(\cla)\bar{\ot}\cle$ and can be written as
\begin{eqnarray}
\label{connection}
    \bar{\partial}_{\cle}=\nabla_{0}+\zeta.
\end{eqnarray}
The $\mathbb{C}$-linear basis for $\Omega^{0,1}(\cla)$ is given by $\{\xi_{\mu\raro\mu-1}\}_{\mu\in V}$ whereas we shall denote the $\mathbb{C}$-linear basis of $\cle$ by $\{\chi_{\mu\raro\nu}\}_{\mu\raro\nu\in E}$. With these notations, for a left-module map $\zeta$, there are scalars $c_{\nu,\nu^{\prime}}^{\mu}\in\mathbb{C}$ such that
\begin{displaymath}
    \zeta(\chi_{\mu\raro\nu})=\sum_{\mu,\nu,\nu^{\prime}\in V:\mu-1\raro\nu^{\prime}\in E}c^{\mu}_{\nu,\nu^{\prime}}\xi_{\mu\raro\mu-1}\ot\chi_{\mu-1\raro\nu^{\prime}}.
\end{displaymath}
As $\Omega^{0,2}(\cla)\bar{\ot}\cle=0$, extension of the operator $\bar{\partial}_{\cle}$ on the Hilbert space $L^{2}(\Omega^{0,1}(\cla)\bar{\ot}\cle)$ becomes the zero map. Thus we consider the operator $\bar{\partial}_{\cle}$ as a $\mathbb{C}$-linear operator in the hilbert space $\clh=L^{2}(\cle)\oplus L^{2}(\Omega^{0,1}(\cla)\bar{\ot}\cle)$ where the hilbert space is equipped with the inner product $\langle,\rangle_{\cle}$. We quickly note the following lemma whose proof is straightforward verification using the definition of the inner product $\langle,\rangle_{\cle}$ and is ommitted.
\begin{lemma}
    \label{onb}
    The hilbert space $\clh$ has the following orthonormal basis:
    \begin{eqnarray}
        \label{onb1}
        \{\sqrt{n}\chi_{\theta\raro\nu}\}_{\theta,\nu\in V:\theta\raro\nu\in E},\{\sqrt{n}\xi_{\lambda+1\raro\lambda}\ot\chi_{\lambda\raro\nu}\}_{\lambda,\nu:\lambda\raro\nu\in E}.
    \end{eqnarray}
\end{lemma}As all the modules involved here are finite dimensional vector spaces, the analytic technicality of whether $\bar{\partial}_{\cle}$ is adjointable or not does not arise so that we can write down the adjoint of the map $\bar{\partial}_{\cle}$. We denote the adjoint by $\bar{\partial}_{\cle}^{\dagger}$. Then using the facts that $\partial_{\cle}^{2}=0=\partial_{\cle}^{\dagger2}$ and $\partial_{\cle}^{\dagger}$ is zero on $\cle$, the twisted Laplacian $\cll$ on $\cle$ is given by 
\begin{displaymath}
    \cll:=(\bar{\partial}_{\cle}+\bar{\partial}_{\cle}^{\dagger})^2=\bar{\partial}_{\cle}^\dagger\bar{\partial}_{\cle}.
\end{displaymath}
As $\bar{\partial}_{\cle}^{\dagger}=\nabla_{0}^{\dagger}+\zeta^{\dagger}$, we shall determine $\nabla_{0}^{\dagger}$ and $\zeta^{\dagger}$ separately. 
\begin{lemma}
\label{nablanotdagger}
For any basis element $\chi_{\mu\raro\nu}\in \cle$, 
\begin{eqnarray}
    \nabla_{0}^{\dagger}(\xi_{\mu+1\raro\mu}\ot\chi_{\mu\raro\nu})=\chi_{\mu\raro\nu}.
\end{eqnarray}
\end{lemma}
\begin{proof}
    For $\chi_{\mu\raro\nu}\in\cle$,
    \begin{align*}
        &\langle\nabla_{0}(\chi_{\mu\raro\nu}),(\xi_{\mu+1\raro\mu}\ot\chi_{\mu\raro\nu})\rangle_{\cle}\\
        &\qquad= \langle(\xi_{\mu+1\raro\mu}\ot\chi_{\mu\raro\nu}),(\xi_{\mu+1\raro\mu}\ot\chi_{\mu\raro\nu})\rangle_{\cle}\\
        &\qquad=\tau\circ g_{\cle}\Big((\xi_{\mu+1\raro\mu}\ot\chi_{\mu\raro\nu}),(\xi_{\mu+1\raro\mu}\ot\chi_{\mu\raro\nu})\Big)\\
        &\qquad= \tau\Big(\ast_{\kappa}(\xi_{\mu+1\raro\mu}\wedge h(\bar{\chi}_{\mu\raro\nu})(\chi_{\mu\raro\nu})\wedge\ast_{\kappa}(\xi_{\mu+1\raro\mu})^{\ast})\Big)\\
        &\qquad=  \tau\Big(\ast_{\kappa}(\xi_{\mu+1\raro\mu}\delta_{\mu}\wedge(i\xi_{\mu\raro\mu+1})\Big)
    \end{align*}
    But $\ast_{\kappa}(i\xi_{\mu+1\raro\mu}\wedge\xi_{\mu\raro\mu+1})$ is the function which takes the value $1$ on the vertex $\mu$ and $0$ on the rest of the vertices. Therefore 
    \begin{displaymath}
      \langle\nabla_{0}(\chi_{\mu\raro\nu}),(\xi_{\mu+1\raro\mu}\ot\chi_{\mu\raro\nu})\rangle_{\cle}=\frac{1}{n}.  
    \end{displaymath}
    For any other edge $\mu^{\prime}\raro\nu^{\prime}$ than $\mu\raro\nu$, similar computation will yield
    \begin{displaymath}
     \langle\nabla_{0}(\chi_{\mu^{\prime}\raro\nu^{\prime}}),(\xi_{\mu+1\raro\mu}\ot\chi_{\mu\raro\nu})\rangle_{\cle}=0.   
    \end{displaymath}
    Again by definition of $\langle,\rangle_{\cle}$, $\langle\chi_{\mu^{\prime}\raro\nu^{\prime}},\chi_{\mu\raro\nu}\rangle_{\cle}=\delta_{\mu,\mu^{\prime}}\delta_{\nu,\nu^{\prime}}\frac{1}{n}$. Consequently by $\mathbb{C}$-linearily of $\nabla_{0}$, for all $\omega\in\cle$,
    \begin{displaymath}
        \langle\nabla_{0}(\omega),(\xi_{\mu+1\raro\mu}\ot\chi_{\mu\raro\nu})\rangle_{\cle}=\langle\omega,\chi_{\mu\raro\nu}\rangle_{\cle}.
    \end{displaymath}
    Hence by non-degeneracy of $\langle,\rangle_{\cle}$, $\nabla_{0}^{\dagger}(\xi_{\mu+1\raro\mu}\ot\chi_{\mu\raro\nu})=\chi_{\mu\raro\nu}$.
\end{proof}

\begin{lemma}
    For any basis element $\chi_{\mu\raro\nu}\in\cle$,
    \begin{eqnarray}
        &&\zeta^{\dagger}(\xi_{\mu+1\raro\mu}\ot\chi_{\mu\raro\nu})=\sum_{}\overline{c^{\mu+1}_{\nu^{\prime},\nu}}\chi_{\mu+1\raro\nu^{\prime}}
    \end{eqnarray}
\end{lemma}
\begin{proof}
 Fix a vertex $\mu+1$. Then for any vertices $\nu^{\prime},\nu$ such that $\chi_{\mu+1\raro\nu^{\prime}},\chi_{\mu\raro\nu}\in\cle$, 
\begin{align*} &\langle\zeta(\chi_{\mu+1\raro\nu^{\prime}}),(\xi_{\mu+1\raro\mu}\ot\chi_{\mu\raro\nu})\rangle_{\cle} \\
&\qquad= \sum c^{\mu+1}_{\nu^{\prime},\nu^{\prime\prime}}\langle\xi_{\mu+1\raro\mu}\ot\chi_{\mu\raro\nu^{\prime\prime}},\xi_{\mu+1\raro\mu}\ot\chi_{\mu\raro\nu}\rangle_{\cle}\\
&\qquad= c^{\mu+1}_{\nu^{\prime},\nu}\tau(\delta_{\mu})\\
&\qquad= \frac{1}{n}c^{\mu+1}_{\nu^{\prime},\nu}.
\end{align*}  
As for any other vertices $\theta, \theta^{\prime}$ such that $\theta\neq \mu+1$ and $\chi_{\theta\raro\theta^{\prime}}\in\cle$,
\begin{displaymath}
    \langle\zeta(\chi_{\theta\raro\theta^{\prime}}),(\xi_{\mu+1\raro\mu}\ot\chi_{\mu\raro\nu})\rangle_{\cle}=0.
\end{displaymath}
But it is easy to see that 
\begin{eqnarray*}
\left\langle\chi_{\theta\raro\nu^{\prime}},\sum\overline{c^{\mu+1}_{\nu^{\prime},\nu}}\chi_{\mu+1\raro\nu^{\prime}}\right\rangle_{\cle}=
\begin{cases}
\frac{1}{n}c^{\mu+1}_{\nu^{\prime},\nu}, &\quad \theta=\mu+1\\
 0, &\quad \theta\neq \mu+1.
\end{cases}
\end{eqnarray*} 
Therefore, by non-degeneracy of the inner product $\langle,\rangle_{\cle}$, we conclude that
\begin{displaymath}
    \zeta^{\dagger}(\xi_{\mu+1\raro\mu}\ot\chi_{\mu\raro\nu})=\sum_{}\overline{c^{\mu+1}_{\nu^{\prime},\nu}}\chi_{\mu+1\raro\nu^{\prime}}.
\end{displaymath}
\end{proof}
Consequently, we have the following easy to see formulae:
\begin{eqnarray}
    \zeta^{\dagger}\zeta(\chi_{\mu\raro\nu})&=&\sum_{\mu-1\raro\nu^{\prime},\mu\raro\nu^{\prime\prime}}c^{\mu}_{\nu,\nu^{\prime}}\overline{c^{\mu}_{\nu^{\prime\prime},\nu^{\prime}}}\chi_{\mu\raro\nu^{\prime\prime}}\\
    \nabla_{0}^{\dagger}\nabla_{0}(\chi_{\mu\raro\nu})&=&\chi_{\mu\raro\nu}\\
    \nabla_{0}^{\dagger}\zeta(\chi_{\mu\raro\nu})&=&\sum c^{\mu}_{\nu,\nu^{\prime}}\chi_{\mu-1\raro\nu^{\prime}}\\
    \zeta^{\dagger}\nabla_{0}(\chi_{\mu\raro\nu})&=&\sum\overline{c^{\mu+1}_{\nu^{\prime},\nu}}\chi_{\mu+1\raro\nu^{\prime}}
\end{eqnarray}
Then for any $\chi_{\mu\raro\nu}\in\cle$,
\begin{eqnarray}
    \cll(\chi_{\mu\raro\nu})&=&\bar{\partial}_{\cle}^{\dagger}\bar{\partial}_{\cle}(\chi_{\mu\raro\nu})\\&=&(\nabla_{0}^{\dagger}+\zeta^{\dagger})(\nabla_{0}+\zeta)(\chi_{\mu\raro\nu}) \ \nonumber \\&=&\chi_{\mu\raro\nu}+\sum c^{\mu}_{\nu,\nu^{\prime}}\chi_{\mu-1\raro\nu^{\prime}}+\sum_{}\overline{c^{\mu+1}_{\nu^{\prime},\nu}}\chi_{\mu+1\raro\nu^{\prime}}\\
    &\qquad+&\sum_{\substack{\mu-1\raro\nu^{\prime},\\ \mu\raro\nu^{\prime\prime}}}c^{\mu}_{\nu,\nu^{\prime}}\overline{c^{\mu}_{\nu^{\prime\prime},\nu^{\prime}}}\chi_{\mu\raro\nu^{\prime\prime}} \nonumber.
\end{eqnarray}
We may call the left module map $\zeta$ a potential for the Laplacian $\cll$ and the coefficients $c^{\mu}_{\nu,\nu^{\prime}}$ the potential coefficients. We shall be interested in the spectrum of the twisted edge Laplacian where all the potential coefficients are $1$. Then the Laplacian $\cll$ acts on the basis elements by the following formula:
\begin{eqnarray}
 \label{laplacian formula}\cll(\chi_{\mu\raro\nu})&=&\Big(1+{\rm deg}(\mu-1)\Big)\chi_{\mu\raro\nu}+\sum_{\mu\raro\nu^{\prime}\in E:\nu\neq\nu^{\prime}}\Big({\rm deg}(\mu-1)\Big)\chi_{\mu\raro\nu^{\prime}} \nonumber \\ 
 &\qquad+&\sum_{\mu-1\raro\nu\in E}\chi_{\mu-1\raro\nu}+\sum_{\mu+1\raro\nu\in E}\chi_{\mu+1\raro\nu},   
\end{eqnarray} where ${\rm deg}(\mu)$ is equal to the number of edges $e$ such that $s(e)=\mu$. For a general 
 $f\in C(E)$ , for an edge $e\in E$,
\begin{eqnarray}
\label{Euclidean}
    \cll(f)(e)=f(e)&+&\sum_{s(e)=s(e^{\prime})}{\rm deg} \big(s(e)-1\big)f(e^{\prime})\\
    &+&\sum_{s(e)=s(e^{\prime\prime})+1}f(e^{\prime\prime})+\sum_{s(e)=s(e^{\prime\prime\prime})-1}f(e^{\prime\prime\prime})\nonumber
\end{eqnarray}
\section{Eigenvalues of twisted edge Laplacians: Examples} In this section we shall see a couple of examples of the twisted edge Laplacians where the Laplacian is considered with respect to the potential coefficients $c^{\mu}_{\nu,\nu^{\prime}}=1$. In this particular case, the Laplacian $\cll$ takes the form (\ref{Euclidean}). We start with a $d$-regular graph.\vspace{0.1in}\\
1. {\bf {$d$-regular graph}}: Let $\mathcal{G}=(V,E)$ be a $d$-regular graph. By a $d$-regular graph we mean that each vertex of the graph has degree $d$. We index the vertices of $\mathcal{G}$ by $1,2,\ldots,n$. We take the following ordered basis of $C(E)$ to write down the twisted edge Laplacian of the graph $\mathcal{G}$:
\begin{displaymath}
    \{\chi_{1\rightarrow\mu}\}_{\mu\in\{1,2,\ldots,n\}},\{\chi_{2\rightarrow\mu}\}_{\mu\in\{1,2,\ldots,n\}},\ldots,\{\chi_{n\rightarrow\mu}\}_{\mu\in\{1,2,\ldots,n\}}.
\end{displaymath}
Note that for any $i=1,2,\ldots,n$, only those $\mu$'s appear in the above list for which $i\raro\mu$ is an edge. By formula (\ref{laplacian formula}), with respect to the above basis of $C(E)$, $\cll$ can be represented by a block matrix of order \(nd\times nd\) such that the diagonal blocks of of order \(d\) are of the form
\[
 \begin{pmatrix}
    d+1 & d & d & \cdots & d \\
    d & d+1 & d &\cdots &d \\
    \vdots & \vdots& \ddots &  & \vdots\\
    d& d &\cdots & & d+1
 \end{pmatrix}
\]
and \(ij\)-th block \(B\) of order \(d\) is of the form
\[ B=(b_{kl})=
 \begin{cases}
  b_{kl}=1 \text{ for all } k,l \text{ iff i and j are adjacent}\\
  b_{kl}=0 \text{ for all } k,l \text{ otherwise}.
 \end{cases}
\]
In this case it can be seen easily that the sum of entries of each row and column of the above matrix are equal to $(d+1)+d(d-1)+2d=(d+1)^{2}$. 
\begin{lemma}
\label{regular}
    For a $d$-regular graph, for any the eigenvalue $\lambda$ of the twisted edge Laplacian $\cll$, we have $0\leq\lambda\leq (d+1)^2$ with $(d+1)^{2}$ being an eigenvalue. 
\end{lemma}
\begin{proof}
 As the sum of each row and each column of the matrix of $\cll$ is equal to $(d+1)^2$, it is trivial to see that $(d+1)^{2}$ is an eigenvalue with eigenvector $(1,1,\ldots,1)$. To show that $|\lambda|\leq (d+1)^2$, we shall apply Gershgorin circle theorem. Take any $i$-th row. Observe that each entry of such a row is a non-negative integer (in fact $d, d+1$ or $1$). If we denote the entries by $a_{ij}$, then $a_{ii}=(d+1)$. Therefore any eigenvalue of $\cll$ lies within the disc $D(a_{ii},R_{i})$ where $R_{i}=\sum_{i\neq j}a_{ij}=(d+1)^2-(d+1)=d(d+1)$. In fact for all $i$, $a_{ii}=(d+1)$ and $R_{i}=d(d+1)$. This shows that any eigenvalue lies in the closed disc $D((d+1),d(d+1))$ by the Gershgorin circle theorem. In particular it shows that for any eigenvalue $\lambda$, $|\lambda|\leq (d+1)+d(d+1)=(d+1)^2$. Now as $\cll$ is a positive definite, $0\leq \lambda\leq (d+1)^2$.
\end{proof}
2. {\bf directed regular $n$-gon}:

\begin{figure}[ht]
    \begin{minipage}[h]{0.5\textwidth}
     \centering
\begin{tikzpicture}
\Vertex[label=$1$,position=above left,shape=circle,size=0.05,color=black]{1}
  \Vertex[x=.7, y=1, label=$2$,position=above,shape=circle, size=0.05,color=black]{2}
  \Vertex[x=2.3,y=1,label=$3$, position=above,shape=circle, size=0.05,color=black]{3}
  \Vertex[x=3,label=$4$,position=above right,shape=circle, size=0.05,color=black]{4}
  \Vertex[,y=-1.5,label=$n$,position=below left,shape=circle, size=0.05,color=black]{8}
  \Vertex[x=3,y=-1.5,label=$5$,position=below right,shape=circle, size=0.05,color=black]{5}
  \Vertex[x=2.3,y=-2.5,position=below,label=$6$,shape=circle, size=0.05,color=black]{6}
  \Vertex[x=.7,y=-2.5,label=$n-1$,position=below,shape=circle, size=0.05,color=black]{7}
  \Edge[Direct](1)(2);
  \Edge[Direct](2)(3);
  \Edge[Direct](3)(4);
  \Edge[Direct](4)(5);
  \Edge[Direct](5)(6);
  \Edge[style={dashed},Direct](6)(7);
  \Edge[Direct](8)(1);
  \Edge[Direct](7)(8);
\end{tikzpicture}
\\
Directed polygon on \(n\)-points
\end{minipage}
\end{figure}

\begin{lemma}
    \label{regulargon}
    The complete list of eigenvalues of the twisted edge Laplacian $\cll$ is given by
    \begin{eqnarray}
        2+2{\rm cos}\frac{2\pi j}{n}, \ 0\leq j\leq n-1,
    \end{eqnarray} where $\omega=e^{\frac{2\pi i}{n}}$ is a primitive \(n\)-th root of unity. In particular, the eigenvalues lie between $0$ and $4$ with $4$ being an eigenvalue. For $n=2m$, $0$ is an eigenvalue whereas for $n=2m+1$, $0$ is not an eigenvalue of $\cll$. 
\end{lemma}
\begin{proof}
    As a regular $n$-gon is $1$-regular, by Lemma \ref{regular}, for any eigenvalue $\lambda$ of $\cll$, $0\leq\lambda\leq (1+1)^{2}=4$ with $4$ being an eigenvalue. In fact note that the matrix $\cll$ with respect to the basis introduced in the beginning of this section takes the following form:
   \begin{center} $\begin{pmatrix}
        2 & 1 & 0 & . & . & . & 0 & 1\\
        1 & 2 & 1 & 0 & . & . & . & 0\\
        0 & 1 & 2 & 1 & 0 & . & . & 0\\
        . & 0 & 1 & 2 & 1 & . & . & 0\\
        .\\
        . \\
        1 & 0 & . & . & . & 0 & 1 & 2
    \end{pmatrix}$\end{center} which is an $n\times n$ circulant matrix. Let $\omega=e^{\frac{2\pi i}{n}}$. Therefore the eigenvalues are given by, for \(\ 0\leq j\leq n-1\),  
    \begin{eqnarray*}
        \lambda_j&=&2+\omega^{j}+\omega^{(n-1)j} \\
        &=& 2+2{\rm cos}\frac{2\pi j}{n}.
    \end{eqnarray*} Now let $n$ be an even integer say $2m$. Then for $j=m$, $\lambda=0$. Hence $0$ is an eigenvalue of $\cll$. In fact, \((1,-1,1,-1,\cdots,1,-1)^{T}\) is an eigenvector for the eigenvalue \(0\). 
    If $n=2m+1$ for some integer $m$, it is easy to see that the list of eigenvalues does not contain $0$. Therefore for an odd integer $n$, the twisted edge Laplacian $\cll$ of a regular directed $n$-gon is invertible. 
\end{proof}
    
\section{A twisted holomorphic Dolbeault Dirac spectral triple}
In this section we briefly discuss Connes' distance between the vertices of a graph $(V,E)$ with respect to the spectral triple obtained from the twisted holomorphic Dolbeault Dirac operator acting on the Hilbert space $\clh=L^{2}(\cle)\oplus L^{2}(\Omega^{0,1}(\cla)\bar{\ot}\cle)$ where $\cle$ is the left $C(V)$-module and the Hilbert space is formed with respect to the inner product $\langle,\rangle_{\cle}$ obtained in Section \ref{edgelaplacian}. The consideration of Dolbeault-Dirac spectral triple and its detailed study (particularly the compact resolvent condition of the Dolbeault-Dirac operator) in the context of quantum projective space appeared in \cite{Das_Somberg}. It is worth mentioning that in our case as the Hilbert space is finite dimensional, the compact resolvent condition is automatic for our spectral triple. Recall the twisted holomorphic Dolbeault-Dirac operator $\cld:=\bar{\partial}_{\cle}+\bar{\partial}_{\cle}^{\dagger}$ acting on $\clh$ where $\bar{\partial}_{\cle}=\nabla_{0}+\zeta$ as in Equation number (\ref{connection}). Recall that $\nabla_{0}$ is the base $(0,1)$-connection and $\zeta:\cle\raro\Omega^{0,1}(\cla)\bar{\ot}\cle$ is some left module map. As the Hilbert space is finite dimensional, denoting $C(V)$ by $\cla$ as usual, $(\cla,\clh,\cld)$ is canonically a spectral triple. $\cla=C(V)$ acts on $\cle$ as well as $\Omega^{0,1}(\cla)\bar{\ot}\cle$ by its left module action. Again as before, we label the vertices $V$ by natural numbers $\mu=1,2,\ldots,n$ and adopt the convention of \cite{Sou_Ati1}. The pure states of $\cla$ correspond to the points in $V$. Let us briskly recall Connes' distance between states on a $C^{\ast}$-algebra. If $(\cla,\clh,\cld)$ is a spectral triple where $\cla\subset\clb(\clh)$ is a $C^{\ast}$-algebra, then for any two states $\phi,\psi\in \cla^{\ast}$, the Connes' distance between $\phi,\psi$ is given by the following formula (see \cite{Connes}):
\begin{displaymath}
    d(\phi,\psi)={\rm sup}_{f\in\cla:||[\cld,f]||\leq 1}|\phi(f)-\psi(f)|.
\end{displaymath}
In the above formula we have supressed the notation of representation of $\cla$ in $\clb(\clh)$ and will continue to do so throughout the rest of the paper. It is known that Connes' distance in particular recovers the geodesic distance between points on a Riemannian manifold for the spinor spectral triple. However, for a general spectral triple, Connes' distance may be infinite at times. In the following lemma $d$ stands for the Connes' distance between two points seen as pure states on $\cla$.
\begin{lemma}
    \label{connes-distance}
    For all $\mu=1,2,\ldots,n$, $d(\mu,\mu+1)= 1$ and the distance between $\mu,\mu+1$ does not depend upon any choice of left module map $\zeta$ to define the Dolbeault-Dirac operator $\cld$.
\end{lemma}
\begin{proof}
     We denote both the Hilbert space norm and the operator norm by $||\cdot||$. Note that $\{\chi_{\mu\raro\nu}\}_{\mu\raro\nu\in E}$ forms an orthogonal set with respect to the inner product $\langle,\rangle_{\cle}$. Moreover $\langle\chi_{\mu\raro\nu},\chi_{\mu\raro\nu}\rangle_{\cle}=\frac{1}{n}$. Therefore $\{\sqrt{n}\chi_{\mu\raro\nu}\}_{\mu\raro\nu\in E}$ forms an orthonormal $\mathbb{C}$-linear basis of $L^{2}(\cle)$ with respect to the inner product $\langle,
     \rangle_{\cle}$. With this observation let $f\in\cla$ such that $||[\cld,f]||\leq 1$. For any $\mu\raro\nu\in E$, $||\sqrt{n}\chi_{\mu\raro\nu}||=1$ in $\clh$. Then $||[\cld,f](\sqrt{n}\chi_{\mu\raro\nu})||\leq||[\cld,f]||\leq 1$. Note that for any left module map $\zeta:\cle\raro\Omega^{0,1}\bar{\ot}\cle$, $[\zeta,f]=0=[\zeta^{\dagger},f]$ on appropriate subspaces. Also for any $\chi_{\mu\raro\nu}$, $\bar{\partial}_{\cle}^{\dagger}(\chi_{\mu\raro\nu})=0$. With these observations using the fact that $\langle\xi_{\mu+1\raro\mu}\ot\chi_{\mu\raro\nu},\xi_{\mu+1\raro\mu}\ot\chi_{\mu\raro\nu}\rangle_{\cle}=\frac{1}{n}$, we get the following
    \begin{equation*}
      \begin{aligned}
        [\cld,f](\sqrt{n}\chi_{\mu\raro\nu})&=\sqrt{n}\Big(f(\mu)\cld(\chi_{\mu\raro\nu})-f.(\xi_{\mu+1\raro\mu}\ot\chi_{\mu\raro\nu})\Big)\\
        &= \sqrt{n}\Big(f(\mu)-f(\mu+1)\Big)\Big(\xi_{\mu+1\raro\mu}\ot\chi_{\mu\raro\nu}\Big).\\
        ||[\cld,f](\sqrt{n}\chi_{\mu\raro\nu})||&=\sqrt\Big\langle\sqrt{n}\Big(f(\mu)-f(\mu+1)\Big)\Big(\xi_{\mu+1\raro\mu}\ot\chi_{\mu\raro\nu}\Big),\\
       &\qquad\sqrt{n}\Big(f(\mu)-f(\mu+1)\Big)\Big(\xi_{\mu+1\raro\mu}\ot\chi_{\mu\raro\nu}\Big)\Big\rangle_{\cle} \\
       &= |f(\mu)-f(\mu+1)|.
      \end{aligned}
    \end{equation*} 
     Therefore for any $f\in\cla$ such that $||[\cld,f]||\leq 1$, $|f(\mu)-f(\mu+1)|\leq 1$. Hence by definition, $d(\mu,\mu+1)\leq 1$ for all $\mu=1,2,\ldots,n$. Now we shall show that the upper bound is actually attained. To that end for a fixed vertex $\mu$, choose a function $g$ on $V$ that takes the value $1$ on $\mu$ and $0$ on $\mu+1$. Recalling the orthonormal basis of $\clh$ from \ref{onb1}, for any element $\Xi=\sum_{\theta,\nu:\theta\raro\nu\in E}c_{\theta}^{\nu}\sqrt{n}\chi_{\theta\raro\nu}+\sum_{\lambda,\nu^{\prime}:\lambda\raro\nu^{\prime}\in E}d_{\lambda}^{\nu^{\prime}}\sqrt{n}\xi_{\lambda+1\raro\lambda}\ot\chi_{\lambda\raro\nu^{\prime}}\in\clh$ such that $||\Xi||\leq 1$, we have
     \begin{equation}
     \label{eqn1}\sum_{\theta,\nu:\theta\raro\nu\in E}|c_{\theta}^{\nu}|^2+\sum_{\lambda,\nu^{\prime}:\lambda\raro\nu^{\prime}\in E}|d_{\lambda}^{\nu^{\prime}}|^2\leq 1.
     \end{equation} 
     Then again using $[\zeta,g]=[\zeta^{\dagger},g]=0$ on appropriate subspaces,
     \begin{align*}
         [\cld,g](\Xi)&=(\nabla_{0}+\nabla_{0}^{\dagger}) g.\left(\sum_{\substack{\theta,\nu \\ \theta\raro\nu\in E}}c_{\theta}^{\nu}\sqrt{n}\chi_{\theta\raro\nu}+\sum_{\substack{\lambda,\nu^{\prime} \\ \lambda\raro\nu^{\prime}\in E}}d_{\lambda}^{\nu^{\prime}}\sqrt{n}\xi_{\lambda+1\raro\lambda}\ot\chi_{\lambda\raro\nu^{\prime}}\right)\\
         &\quad-g.(\nabla_{0}+\nabla_{0}^{\dagger})\left(\sum_{\substack{\theta,\nu \\ \theta\raro\nu\in E}}c_{\theta}^{\nu}\sqrt{n}\chi_{\theta\raro\nu}+\sum_{\substack{\lambda,\nu^{\prime} \\ \lambda\raro\nu^{\prime}\in E}}d_{\lambda}^{\nu^{\prime}}\sqrt{n}\xi_{\lambda+1\raro\lambda}\ot\chi_{\lambda\raro\nu^{\prime}}\right)\\
         &= (\nabla_{0}+\nabla_{0}^{\dagger})\left(\sum_{\substack{\nu:\\ \mu\raro\nu\in E}}c_{\mu}^{\nu}\sqrt{n}\chi_{\mu\raro\nu}+\sum_{\substack{\nu^{\prime} \\ \mu-1\raro\nu^{\prime}\in E}}d_{\mu-1}^{\nu^{\prime}}\sqrt{n}\xi_{\mu\raro\mu-1}\ot\chi_{\mu-1\raro\nu^{\prime}}\right)\\
         &\quad- g.\left(\sum_{}c^{\nu}_{\theta}\sqrt{n}\xi_{\theta+1\raro\theta}\ot\chi_{\theta\raro\nu}+\sum_{}d_{\lambda}^{\nu^{\prime}}\sqrt{n}\chi_{\lambda\raro\nu^{\prime}}\right)\\
         &= \left(\sum_{\nu:\mu\raro\nu\in E}c_{\mu}^{\nu}\sqrt{n}\xi_{\mu+1\raro\mu}\ot \chi_{\mu\raro\nu}+\sum_{\nu^{\prime}:\mu-1\raro\nu^{\prime}\in E}d_{\mu-1}^{\nu^{\prime}}\sqrt{n}\chi_{\mu-1\raro\nu^{\prime}}\right)\\
         &\quad- \left(\sum_{\nu:\mu-1\raro\nu\in E}c^{\nu}_{\mu-1}\sqrt{n}\xi_{\mu\raro\mu-1}\ot\chi_{\mu-1\raro\nu}+\sum_{\nu^{\prime}:\mu\raro\nu^{\prime}\in E}d_{\mu}^{\nu^{\prime}}\sqrt{n}\chi_{\mu\raro\nu^{\prime}}\right)
     \end{align*} 
Therefore \begin{displaymath}||[\cld,g](\Xi)||^2=\sum_{\substack{\nu: \\ \mu\raro\nu\in E}}|c_{\mu}^{\nu}|^{2}+\sum_{\substack{\nu:\\ \mu-1\raro\nu\in E}}|c_{\mu-1}^{\nu}|^{2}+\sum_{\substack{\nu:\\ \mu\raro\nu\in E}}|d_{\mu}^{\nu}|^{2}+\sum_{\substack{\nu:\\ \mu-1\raro\nu\in E}}|d_{\mu-1}^{\nu}|^{2}.\end{displaymath} Then by inequality (\ref{eqn1}), 
     \begin{displaymath}
         ||[\cld,g](\Xi)||\leq 1,
     \end{displaymath} which implies that $||[\cld,g]||\leq 1$. But $|g(\mu)-g(\mu+1)|=1$. Hence $d(\mu,\mu+1)=1$ for all $\mu=1,2,\ldots,n$. From the proof it is clear that the distance does not depend upon any choice of left module map $\zeta$ to define the holomorphic Dolbeault-Dirac operator $\cld$.
\end{proof}
We have the following easy to see corollary:
\begin{corollary}
    \label{diameter}
    For any two points $\mu,\nu\in V$, $d(\mu,\nu)\leq[\frac{n}{2}]$. In particular, $V$ has finite diameter.
\end{corollary}
{\bf Acknowledgement}: The first author is supported by SERB MATRICS grant (grant number MTR/2022/000515), Government of India. The first author would like to thank Md. Ali Zinna for fruitful discussion, particularly concerning the proof of Lemma \ref{Hermitian}. 

\end{document}